\documentclass{amsart}

\usepackage{amssymb}

\newtheorem{theorem}{Theorem}[section]
\newtheorem{lemma}[theorem]{Lemma}

\theoremstyle{definition}
\newtheorem{definition}[theorem]{Definition}

\newtheorem{corollary}[theorem]{Corollary}
\newtheorem{proposition}[theorem]{Proposition}

\theoremstyle{remark}
\newtheorem{remark}[theorem]{Remark}

\numberwithin{equation}{section}
\begin{document}
%%%%%%%%%%%%%%%%%%%%%%%%%%%%%%%%%%%%%%%%%%%%%%%%%%%%%%%%%%%%%%%%%%%%%%%%%%%%

\title{The angle of an operator and range - kernel complementarity}

\author{Dimosthenis Drivaliaris}
\address{Department of Financial and Management Engineering\\
University of the Aegean\\
Kountourioti 45\\
82100 Chios\\
Greece}
\email{d.drivaliaris@fme.aegean.gr}
\author{Nikos Yannakakis}
\address{Department of Mathematics\\
School of Applied Mathematics and Natural Sciences\\
National Technical University of Athens\\
Iroon Polytexneiou 9\\
15780 Zografou\\
Greece}
\email{nyian@math.ntua.gr}


\date{}

\dedicatory{}

\commby{}

\begin{abstract}
We show that if the angle of a bounded linear operator on a Banach space, with closed range and closed sum of its range and kernel, is less than $\pi$, then its range and kernel are complementary. In finite dimensions and up to scalar multiples this simple geometric property characterizes operators with complementary range and kernel. Applying our result we get simple proofs of two known facts concerning eigenvalues lying in the boundary of the numerical range. For an operator on a Hilbert space we present a sufficient condition for range-kernel complementarity, involving the distance of the boundary of the numerical range from the origin. Finally, we discuss some properties of  operators whose spectrum does not intersect all rays emanating from the origin and show that in a Banach space which is uniformly convex and has uniformly convex dual such operators are surjective if and only if they are injective.
\end{abstract}
%%%%%%%%%%%%%%%%%%%%%%%%%%%%%%%%%%%%%%%%%%%%%%%%%%%%%%%%%%%%%%%%%%%%%%%%%%%%

\maketitle

\section{Introduction}
%%%%%%%%%%%%%%%%%%%%%%%%%%%%%%%%%%%%%%%%%%%%%%%%%%%%%%%%%%%%%%%%%%%%%%%%%%%%

Let $X$ be a complex Banach space and $A:X\rightarrow X$ be a bounded linear operator. We will denote the range of $A$ by $R(A)$ and the kernel of $A$ by $N(A)$. Our aim in this paper is to show that a simple geometric property of the operator $A$, namely that its angle (see (\ref{cosine}) and (\ref{angle}) in the following section) is less than $\pi$, implies that
\begin{equation}
\label{intro}
X=R(A)\oplus N(A),
\end{equation}
provided that $A$ has closed range and $R(A)+N(A)$ is closed. Note that the last hypothesis is unnecessary if $X$ is a Hilbert space (see Theorem \ref{hilbert}).
%%%%%%%%%%%%%%%%%%%%%%%%%%%%%%%%%%%%%%%%%%%%%%%%%%%%%%%%%%%%%%%%%%%%%%%%%%%%

An immediate consequence of the above is that bounded accretive operators either satisfy (\ref{intro}) or at least have ascent equal to 1 (see the following section for the definition of the ascent of an operator). This observation allows us to give alternative proofs of two known results, the first by N. Nirschl and H. Schneider \cite{nirschl}, \cite[Theorem 10.10]{bonsal} and the second by A. M. Sinclair \cite[Proposition 3]{sinclair}, concerning eigenvalues lying in the boundary of the numerical range of an operator.
%%%%%%%%%%%%%%%%%%%%%%%%%%%%%%%%%%%%%%%%%%%%%%%%%%%%%%%%%%%%%%%%%%%%%%%%%%%%

We also show that if $X$ is a strictly convex, finite dimensional Banach space and $A:X\rightarrow X$ is a linear operator, then (\ref{intro}) holds if and only if there exists $0\neq t\in\mathbb{C}$ such that the angle of $tA$ is less than $\pi$.
%%%%%%%%%%%%%%%%%%%%%%%%%%%%%%%%%%%%%%%%%%%%%%%%%%%%%%%%%%%%%%%%%%%%%%%%%%%%

Finally, we present two applications. The first is a sufficient condition, involving the distance of the boundary of the numerical range from the origin, for (\ref{intro}) to hold. The second exploits some properties of operators whose spectrum does not intersect all rays emanating from the origin. In particular we show that if $X$ and $X^*$ are uniformly convex, then such operators are surjective if and only if they are injective.
%%%%%%%%%%%%%%%%%%%%%%%%%%%%%%%%%%%%%%%%%%%%%%%%%%%%%%%%%%%%%%%%%%%%%%%%%%%%

\section{Preliminaries}
%%%%%%%%%%%%%%%%%%%%%%%%%%%%%%%%%%%%%%%%%%%%%%%%%%%%%%%%%%%%%%%%%%%%%%%%%%%%

In what follows, $X$ is a complex Banach space, $\|\cdot\|$ is its norm, $X^*$ is its dual and $\langle\cdot\,,\cdot\rangle$ is the duality product. Throughout we assume that $X$ is equipped with a semi-inner product $[\cdot\,,\cdot]$ compatible with its norm. If $X$ is a Hilbert space, then we will denote the inner product of $X$ by $\langle\cdot,\cdot\rangle$.
%%%%%%%%%%%%%%%%%%%%%%%%%%%%%%%%%%%%%%%%%%%%%%%%%%%%%%%%%%%%%%%%%%%%%%%%%%%%

By $J:X\rightarrow 2^{X^*}$ we denote the duality map of $X$, which is defined by
\begin{equation*}
J(x)=\left\{x^*\in X^*\,:\,\langle x^*,x\rangle=\| x\|^2\text{ and }\| x^*\|=\| x\|\right\},
\end{equation*}
for all $x\in X$. It is easy to see that, for every $x\in X$, $J(x)$ is a non-empty $w^\ast$-compact subset of $X^\ast$. Moreover, if $x, y\in X$, then, for some $x^\ast\in J(x)$,
\begin{equation*}
[y,x]=\langle x^\ast,y\rangle.
\end{equation*}
Note that this implies that if $J$ is single-valued, as is the case if $X^\ast$ is strictly convex, then the semi-inner product of $X$ is unique.
%%%%%%%%%%%%%%%%%%%%%%%%%%%%%%%%%%%%%%%%%%%%%%%%%%%%%%%%%%%%%%%%%%%%%%%%%%%%

Before we move on, we would like to discuss a continuity property of $[\cdot\,,\cdot]$, which we will need for the proof of our main result.
%%%%%%%%%%%%%%%%%%%%%%%%%%%%%%%%%%%%%%%%%%%%%%%%%%%%%%%%%%%%%%%%%%%%%%%%%%%%

Recall that if $Y_1$ and $Y_2$ are topological spaces and $F:Y_1\rightarrow 2^{Y_2}$ is a multifunction, then $F$ is called upper semicontinuous at $y\in Y_1$ if for every open subset $V$ of $Y_2$ containing $F(y)$, there exists an open neighborhood $U$ of $y$ such that $F(U)\subseteq V$ \cite[Definition 1.2.3 and Remark 1.2.4]{pap}.
%%%%%%%%%%%%%%%%%%%%%%%%%%%%%%%%%%%%%%%%%%%%%%%%%%%%%%%%%%%%%%%%%%%%%%%%%%%%

\begin{lemma}
\label{dualitymap}
The duality map $J$ from $X$ equipped with the norm topology into $X^\ast$ equipped with the $w^\ast$-topology is upper semicontinuous.
\end{lemma}
%%%%%%%%%%%%%%%%%%%%%%%%%%%%%%%%%%%%%%%%%%%%%%%%%%%%%%%%%%%%%%%%%%%%%%%%%%%%

\begin{proof}
It is easy to see that the graph of $J$ from $X$ equipped with the norm topology into $X^\ast$ equipped with the $w^\ast$-topology is closed. That together with the fact that $J$ is locally compact (i.e.\ that, for each $x\in X$, there exists an open neighborhood $U$ of $x$ such that $\overline{J(U)}^{w^\ast}$ is $w^\ast$-compact) is equivalent, by \cite[Proposition 1.2.23]{pap}, to $J$ being upper semicontinuous.
\end{proof}
%%%%%%%%%%%%%%%%%%%%%%%%%%%%%%%%%%%%%%%%%%%%%%%%%%%%%%%%%%%%%%%%%%%%%%%%%%%%

\begin{remark}
Note that in \cite[Theorem 4.3 and Definition 4.2]{cudia} it was shown that $J$ from the unit sphere of $X$ equipped with the norm topology into the unit sphere of $X^\ast$ equipped with the $w^\ast$-topology has closed graph.
\end{remark}
%%%%%%%%%%%%%%%%%%%%%%%%%%%%%%%%%%%%%%%%%%%%%%%%%%%%%%%%%%%%%%%%%%%%%%%%%%%%

Using Lemma \ref{dualitymap} we can get the following continuity result for $[\cdot\,,\cdot]$.
%%%%%%%%%%%%%%%%%%%%%%%%%%%%%%%%%%%%%%%%%%%%%%%%%%%%%%%%%%%%%%%%%%%%%%%%%%%%

\begin{proposition}
\label{semi}
Let $(x_n)$ be a sequence in $X$ with $\|x_n\|=1$, for all $n\in\mathbb N$, and $x\in X$ such that $x_n\rightarrow x$. Then there exists a subsequence $(x_{n_k})$ of $(x_n)$ such that
\begin{equation*}
\lim_{k\rightarrow\infty} [x,x_{n_k}]=1.
\end{equation*}
\end{proposition}
%%%%%%%%%%%%%%%%%%%%%%%%%%%%%%%%%%%%%%%%%%%%%%%%%%%%%%%%%%%%%%%%%%%%%%%%%%%%

\begin{proof}
By what we said after the definition of the duality map, for each $n\in \mathbb N$, there exists $x_n^\ast\in J(x_n)$ such that
\begin{equation*}
[x,x_n]=\langle x_n^\ast,x\rangle.
\end{equation*}

We will prove that $(x_n^\ast)$ has a subsequence converging to some $x^\ast\in J(x)$ with respect to the $w^\ast$-topology.

Assume the contrary, i.e.\ that, for every $x^\ast\in J(x)$, there exists a $w^\ast$-open neighborhood $V(x^\ast)$ of $x^\ast$ and $N_{x^\ast}\in \mathbb N$ such that $x_n^\ast\notin V(x^\ast)$, for all $n\geq N_{x^\ast}$.

Obviously
\begin{equation*}
\bigcup_{x^\ast\in J(x)} V(x^\ast)
\end{equation*}
is an $w^\ast$-open cover of $J(x)$ and hence, since $J(x)$ is $w^\ast$-compact, there exists a finite subcover
\begin{equation*}
V=\bigcup_{i=1}^KV(x^\ast_i)
\end{equation*}
of $J(x)$. Obviously, if $N_1=\max_{1\leq i\leq K}N_{x^\ast_i}$, then
\begin{equation}
\label{lem1}
x_n^\ast\notin V,\text{ for all }n\geq N_1.
\end{equation}

Since, by Lemma \ref{dualitymap}, $J$ is upper semicontinuous from $X$ equipped with the norm topology into $X^\ast$ equipped with the $w^\ast$-topology, there exists an open neighborhood $U$ of $x$ such that
\begin{equation*}
J(U)\subseteq V.
\end{equation*}
But, since $x_n\rightarrow x$, there exists $N_2\in \mathbb N$ such that $x_n\in U$, for all $n\geq N_2$. So 
\begin{equation}
\label{lem2}
x_n^\ast\in J(x_n)\subseteq V,\text{ for all }n\geq N_2.
\end{equation}

Obviously combining (\ref{lem1}) and (\ref{lem2}) we get a contradiction and thus there exists a subsequence $(x_{n_k}^\ast)$ of $(x^\ast_n)$ such that
\begin{equation*}
x_{n_k}^\ast\stackrel{w^\ast}{\rightarrow} x^\ast\in J(x).
\end{equation*}
So
\begin{equation*}
\lim_{k\rightarrow \infty}[x,x_{n_k}]=\lim_{k\rightarrow \infty}\langle x_{n_k}^\ast,x\rangle=\langle x^\ast, x\rangle=1.
\end{equation*}
\end{proof}
%%%%%%%%%%%%%%%%%%%%%%%%%%%%%%%%%%%%%%%%%%%%%%%%%%%%%%%%%%%%%%%%%%%%%%%%%%%%

\begin{remark}
The proof of Proposition \ref{semi} is based on \cite[Proposition I.2.19]{pap}.
\end{remark}
%%%%%%%%%%%%%%%%%%%%%%%%%%%%%%%%%%%%%%%%%%%%%%%%%%%%%%%%%%%%%%%%%%%%%%%%%%%%

Let $A:X\rightarrow X$ be a bounded linear operator. The ascent $\alpha(A)$ of $A$ is the smallest positive integer $k$ for which $N(A^k)=N(A^{k+1})$. If no such integer exists, then $\alpha(A)=\infty$. The descent $\delta(A)$ of $A$ is the smallest positive integer $k$ for which $R(A^k)=R(A^{k+1})$. If no such integer exists, then $\delta(A)=\infty$. Recall (see \cite[\textsection2.2]{abramovich}, \cite[pp.\ 26--29]{marek1}, \cite[Chapter 13]{marek}, \cite[V.6]{taylor}) that the following are equivalent:
\begin{itemize}
\item[(i)] $ X=R(A^k)\oplus N(A^k),$
\item[(ii)] $ \alpha(A), \delta(A)\leq k,$
\item[(iii)] There exists a bounded linear operator $A^{d}:X\rightarrow X$ (called the Drazin inverse of $A$), satisfying
\begin{equation*}
A^{k+1}A^{d}=A^k, \,A^{d}AA^{d}=A^{d}\text{ and }AA^{d}=A^{d}A.
\end{equation*}
\end{itemize}
If one of (i), (ii), (iii) holds, then $R(A^k)$ is closed and
\begin{equation*}
\alpha(A)=\delta(A).
\end{equation*}
In that case we call the common value of $\alpha(A)$ and $\delta(A)$ the Drazin index of $A$ and denote it by $i(A)$. We have that 
\begin{equation*}
R(A^k)\cap N(A^k)=\{0\}
\end{equation*}
if and only if
\begin{equation*}
\alpha(A)\leq k.
\end{equation*}
Finally, note that, for $k=1$, $A^d$ is a commuting generalized inverse of $A$ (the group inverse $A^{\#}$).
%%%%%%%%%%%%%%%%%%%%%%%%%%%%%%%%%%%%%%%%%%%%%%%%%%%%%%%%%%%%%%%%%%%%%%%%%%%%

The cosine of a linear operator $A:X\rightarrow X$ with respect to $[\cdot, \cdot]$ is defined by
\begin{equation}
\label{cosine}
\cos A=\inf\left\lbrace\frac{Re[Ax, x]}{\|Ax\|\,\|x\|}\,:\,x\notin N(A)\right\rbrace.
\end{equation}
This concept was introduced by K. Gustafson in \cite{gus1}. Using (\ref{cosine}) one can define the angle
$\phi(A)$ of $A$ by
\begin{equation}
\label{angle}
\phi(A)=\arccos(\cos A).
\end{equation}
The angle $\phi(A)$ of $A$ has an obvious geometric interpretation; it measures the maximum (real) turning effect of $A$.
%%%%%%%%%%%%%%%%%%%%%%%%%%%%%%%%%%%%%%%%%%%%%%%%%%%%%%%%%%%%%%%%%%%%%%%%%%%%

Let $A:X\rightarrow X$ be a bounded linear operator. The numerical range $W(A)$ of $A$ corresponding to $[\cdot,\cdot]$ is defined by
\begin{equation*}
W(A)=\left\{[Ax,x]\,:\,\|x\|=1\right\},
\end{equation*}
whereas its spatial numerical range is defined by
\begin{equation*}
V(A)=\left\{\langle x^\ast, Ax\rangle\,:\,x^\ast\in J(x),\|x\|=1\right\}.
\end{equation*}
Recall that (see \cite[Theorems 9.4 and 9.8]{bonsal})
\begin{equation*}
\partial\,\overline{co}\,W(A)=\partial\,\overline{co}\,V(A),
\end{equation*}
where by $\partial\,\overline{co}\,S$ we denote the boundary of the closed convex hull of a set $S$.
%%%%%%%%%%%%%%%%%%%%%%%%%%%%%%%%%%%%%%%%%%%%%%%%%%%%%%%%%%%%%%%%%%%%%%%%%%%%

If $A, B:X\rightarrow X$ are bounded linear operators with $R(A)$ and $R(B)$ closed, then $R(AB)$ is closed if and only if $N(A)+R(B)$ is closed (see \cite[Corollary 1]{nikaido}). In particular, if $A:X\rightarrow X$ is a bounded linear operator with $R(A)$ closed, then $R(A^2)$ is closed if and only if $N(A)+R(A)$ is closed.
%%%%%%%%%%%%%%%%%%%%%%%%%%%%%%%%%%%%%%%%%%%%%%%%%%%%%%%%%%%%%%%%%%%%%%%%%%%%

A linear operator $A:X\rightarrow X$ is called accretive if there exists some semi-inner product $[\cdot,\cdot]$ such that
\begin{equation*}
Re[Ax, x]\geq 0,\text{ for all }x\in X.
\end{equation*}
%%%%%%%%%%%%%%%%%%%%%%%%%%%%%%%%%%%%%%%%%%%%%%%%%%%%%%%%%%%%%%%%%%%%%%%%%%%%

Finally, recall that if $M, N$ are closed subspaces of $X$ such that $M\nsubseteq N$, then
\begin{equation*}
\gamma(M,N)=\inf_{x\in M, x\notin N}\frac{\mathrm{dist}(x,N)}{\mathrm{dist}(x,M\cap N)}
\end{equation*}
(see \cite[p. 219]{Kato}). We have that $\gamma(M,N)>0$ if and only if $M+N$ is closed \cite[Theorem 4.2, p. 219]{Kato}.
%%%%%%%%%%%%%%%%%%%%%%%%%%%%%%%%%%%%%%%%%%%%%%%%%%%%%%%%%%%%%%%%%%%%%%%%%%%%

\section{Main result}
%%%%%%%%%%%%%%%%%%%%%%%%%%%%%%%%%%%%%%%%%%%%%%%%%%%%%%%%%%%%%%%%%%%%%%%%%%%%

We begin with an important property of the angle $\phi(A)$.
%%%%%%%%%%%%%%%%%%%%%%%%%%%%%%%%%%%%%%%%%%%%%%%%%%%%%%%%%%%%%%%%%%%%%%%%%%%%

\begin{proposition}
\label{zero}
Let $A:X\rightarrow X$ be a linear operator. If
\begin{equation*}
R(A)\cap N(A)\neq\left\{0\right\},
\end{equation*}
then $\phi(A)=\pi$.
\end{proposition}
%%%%%%%%%%%%%%%%%%%%%%%%%%%%%%%%%%%%%%%%%%%%%%%%%%%%%%%%%%%%%%%%%%%%%%%%%%%%

\begin{proof}
The hypothesis
\begin{equation*}
R(A)\cap N(A)\neq\left\{0\right\}
\end{equation*}
implies that there exists $z\in X$ with $\|z\|=1$, $Az\neq0$ and $Az\in N(A)$.

It is easy to see that, for each $n\in\mathbb N$ with $n\geq 2$, there exists $t_n>0$ such that for
\begin{equation*}
x_n=\frac{1}{n}\, z-t_n\,\frac{Az}{\|Az\|}
\end{equation*}
we have that $\|x_n\|=1$. Since $\|x_n\|=1$, for all $n\geq 2$, $\displaystyle{\lim_{n\rightarrow \infty}}t_n=1$, and so
\begin{equation*}
\lim_{n\rightarrow \infty}x_n=-\frac{Az}{\|Az\|}\,.
\end{equation*}
Thus, by Proposition \ref{semi}, we can find a subsequence $(x_{n_k})$ of $(x_{n})$ such that
\begin{equation*}
\lim_{k\rightarrow \infty}\frac{[Az, x_{n_k}]}{\|Az\|}=-1.
\end{equation*}
Since $Az\in N(A)$,
\begin{equation*}
\frac{Re[Ax_n, x_n]}{\|Ax_n\|}=\frac{Re\left[Az,x_n\right]}{\|Az\|},\text{ for all }n\geq 2,
\end{equation*}
and so
\begin{equation*}
\lim_{k\rightarrow \infty}\frac{Re[Ax_{n_k}, x_{n_k}]}{\|Ax_{n_k}\|}=-1,
\end{equation*}
which implies that $\phi(A)=\pi$.
\end{proof}
%%%%%%%%%%%%%%%%%%%%%%%%%%%%%%%%%%%%%%%%%%%%%%%%%%%%%%%%%%%%%%%%%%%%%%%%%%%%

\begin{remark}
\label{rem}
(i) It is easy to see that the above remains true for a not everywhere defined linear operator.\\
(ii) The converse of Proposition \ref{zero} does not hold. To see that let $A=-I_X$. Then $\phi(A)=\pi$ and $R(A)\cap N(A)=\left\{0\right\}$.\\
(iii) Proposition \ref{zero} tells us that $\phi(A)<\pi$ implies that $\alpha (A)\leq 1$.\\
(iv) By (iii)  $\phi(A^k)<\pi$ implies that $\alpha (A^k)\leq 1$ and so $\alpha (A)\leq k$. Note that $\phi(A^k)$ and $\phi(A)$ are not related.\\
(v) If $X$ is finite dimensional, then by Proposition \ref{zero} we get  that $\phi(A)<\pi$ implies that $X=R(A)\oplus N(A)$ and so $i(A)\leq 1$.\\
(vi) If $\varphi(tA)<\pi$, for some $0\neq t\in\mathbb{C}$, then $R(tA)\cap N(tA)=\left\{0\right\}$. Thus, since $R(tA)=R(A)$ and $N(tA)=N(A)$, we get that $R(A)\cap N(A)=\left\{0\right\}$.
\end{remark}
%%%%%%%%%%%%%%%%%%%%%%%%%%%%%%%%%%%%%%%%%%%%%%%%%%%%%%%%%%%%%%%%%%%%%%%%%%%%

Using Proposition \ref{zero} we can get the following result about accretive operators.
%%%%%%%%%%%%%%%%%%%%%%%%%%%%%%%%%%%%%%%%%%%%%%%%%%%%%%%%%%%%%%%%%%%%%%%%%%%%

\begin{corollary}
\label{accretive}
If $A$ is an accretive linear operator, then $\alpha(A)\leq1$.
\end{corollary}
%%%%%%%%%%%%%%%%%%%%%%%%%%%%%%%%%%%%%%%%%%%%%%%%%%%%%%%%%%%%%%%%%%%%%%%%%%%%

\begin{proof}
Obviously $\phi(A)\leq\frac{\pi}{2}<\pi$ and so, by Proposition \ref{zero}, $\alpha(A)\leq1$.
\end{proof}
%%%%%%%%%%%%%%%%%%%%%%%%%%%%%%%%%%%%%%%%%%%%%%%%%%%%%%%%%%%%%%%%%%%%%%%%%%%%

We can now prove our main result.
%%%%%%%%%%%%%%%%%%%%%%%%%%%%%%%%%%%%%%%%%%%%%%%%%%%%%%%%%%%%%%%%%%%%%%%%%%%%

\begin{theorem}
\label{theorem}
Let $A:X\rightarrow X$ be a bounded linear operator with closed range such that $R(A)+N(A)$ is closed. If $\phi(A)<\pi$, then
\begin{equation*}
X=R(A)\oplus N(A).
\end{equation*}
\end{theorem}
%%%%%%%%%%%%%%%%%%%%%%%%%%%%%%%%%%%%%%%%%%%%%%%%%%%%%%%%%%%%%%%%%%%%%%%%%%%%

\begin{proof}
Assume first that $N(A)=\left\{0\right\}$. Then $\phi(A)<\pi$ implies that there exists $\delta>0$ such that
\begin{equation*}
\frac{Re[ Ax, x]}{\|Ax\|\,\|x\|}\geq-1+\delta, \text{ for all }x\neq 0,
\end{equation*}
and so
\begin{equation}
\label{eq1}
Re[Ax, x]+\|Ax\|\,\|x\|\geq \delta\|Ax\|\,\|x\|, \text{ for all }x\in X.
\end{equation}
On the other hand, since $R(A)$ is closed, there exists $c>0$ such that
\begin{equation}
\label{eq1a}
\|Ax\|\geq c\|x\|, \text{ for all }x\in X.
\end{equation}
By (\ref{eq1}) and (\ref{eq1a}) we get that there exists $C>0$ such that
\begin{equation*}
Re[Ax, x]+\|Ax\|\|x\|\geq C\|x\|^2, \text{ for all }x\in X.
\end{equation*}
So we may apply \cite[Theorem 2.14]{drivyann} and get that $R(A)=X$.

For the general case, first note that, by Proposition \ref{zero}, $\phi(A)<\pi$ implies that $\alpha(A)\leq1$. Assume now that $\delta(A)>1$ and hence $R(A^2)\subsetneqq R(A)$. Then
\begin{equation*}
A|_{R(A)}:R(A)\rightarrow R(A)
\end{equation*}
is 1-1 but not onto. Moreover, by what we said in the Preliminaries, our assumption that $R(A)+N(A)$ is closed implies that
\begin{equation*}
R(A|_{R(A)})=R(A^2)
\end{equation*}
is closed. Hence, by the first part of the proof, we get that
\begin{equation*}
\phi(A|_{R(A)})=\pi.
\end{equation*}
Since $\phi(A|_{R(A)})\geq \phi(A)$, we get a contradiction, and so $\delta(A)\leq1$. Therefore
\begin{equation*}
X=R(A)\oplus N(A).
\end{equation*}
\end{proof}
%%%%%%%%%%%%%%%%%%%%%%%%%%%%%%%%%%%%%%%%%%%%%%%%%%%%%%%%%%%%%%%%%%%%%%%%%%%%

\begin{remark}
(i) \cite[Theorem 2.14]{drivyann} used in the above proof is the Banach space version of a result by J. Saint-Raymond in \cite{SaintRaymond} where some generalized versions of the Lax-Milgram theorem were proved, answering a question posed by B. Ricceri in \cite{Ricceri}.\\
(ii) The result of Theorem \ref{theorem} is not true if $A$ does not have closed range. To see that let $A:l^2(\mathbb N)\rightarrow l^2(\mathbb N)$ with $A((x_n))=(\frac{1}{n}\,x_n)$. Then $A$ does not have closed range, $\phi(A)<\pi$ and $R(A)+N(A)\neq X$.\\
(iii) What we said in Remark \ref{rem}(ii) shows that $X=R(A)\oplus N(A)$ does not imply that $\phi(A)<\pi$.\\
(v) Theorem \ref{theorem} tells us that if $A$ has closed range and $R(A)+N(A)$ is closed, then $\phi(A)<\pi$ implies that $i(A)\leq 1$.\\
(vi) By (v) if $R(A^k)$ and $R(A^k)+N(A^k)$ are closed, then $\phi(A^k)<\pi$ implies that $i(A^k)\leq 1$ and so $i(A)\leq k$.\\
(vii) If $A$ is a bounded linear operator with closed range, $R(A)+N(A)$ is closed and $\varphi(tA)<\pi$, for some $0\neq t\in\mathbb{C}$, then $X=R(A)\oplus N(A)$.
\end{remark}
%%%%%%%%%%%%%%%%%%%%%%%%%%%%%%%%%%%%%%%%%%%%%%%%%%%%%%%%%%%%%%%%%%%%%%%%%%%%

Returning to accretive operators we have the next corollary.
%%%%%%%%%%%%%%%%%%%%%%%%%%%%%%%%%%%%%%%%%%%%%%%%%%%%%%%%%%%%%%%%%%%%%%%%%%%%

\begin{corollary}
\label{accretive1}
If $A$ is an accretive, bounded linear operator with closed range, then $X=R(A)\oplus N(A)$.
\end{corollary}
%%%%%%%%%%%%%%%%%%%%%%%%%%%%%%%%%%%%%%%%%%%%%%%%%%%%%%%%%%%%%%%%%%%%%%%%%%%%

\begin{proof}
Using \cite[Proposition 1]{sinclair} we get that if $A$ is an accretive, bounded linear operator with closed range, then $R(A)+N(A)$ is closed. The result then follows from Theorem \ref{theorem}.
\end{proof}
%%%%%%%%%%%%%%%%%%%%%%%%%%%%%%%%%%%%%%%%%%%%%%%%%%%%%%%%%%%%%%%%%%%%%%%%%%%%

We will now show that if $X$ is a Hilbert space, then the hypothesis that $R(A)+N(A)$ is closed can be dropped.
%%%%%%%%%%%%%%%%%%%%%%%%%%%%%%%%%%%%%%%%%%%%%%%%%%%%%%%%%%%%%%%%%%%%%%%%%%%%

\begin{theorem}
\label{hilbert}
Let $X$ be a Hilbert space and $A:X\rightarrow X$ be a bounded linear operator with closed range. If $\phi(A)<\pi$, then
\begin{equation*}
X=R(A)\oplus N(A).
\end{equation*}
\end{theorem}
%%%%%%%%%%%%%%%%%%%%%%%%%%%%%%%%%%%%%%%%%%%%%%%%%%%%%%%%%%%%%%%%%%%%%%%%%%%%

\begin{proof}
By Theorem \ref{theorem} we need to prove that $R(A)+N(A)$ is closed. 

If we assume that this is not true, then, by what we said in the Preliminaries, $R(A^2)$ is not closed and so $R(A|_{R(A)})$ is not closed. Thus we can find a sequence $(z_n)$ in $X$ such that 
\begin{equation}
\label{n}
\|Az_n\|=1
\end{equation}
and
\begin{equation}
\label{nsquare}
\|A^2 z_n\|\leq \frac{1}{n^2}\,,
\end{equation}
for all $n\in\mathbb{N}$. Without loss of generality we may assume that $(z_n)$ lies in $N(A)^{\perp}$. Since $R(A|_{N(A)^{\perp}})=R(A)$, $A:N(A)^{\perp}\rightarrow X$ is 1-1 with closed range and so there exists $c>0$ such that
\begin{equation*}
\|A z_n\|\geq c\,\|z_n\|\text{ for all }n\in\mathbb{N}.
\end{equation*} 
Thus $(z_n)$ is bounded and so
\begin{equation}
\label{z}
\lim_{n\rightarrow \infty}\left\|\frac{1}{n}\,z_n\right\|=0
\end{equation}
and in particular there exists $n_0\in\mathbb{N}$ such that $\|\frac{1}{n}\,z_n\|<1$, for all $n\geq n_0$. 

As in the proof of Proposition \ref{zero}, for each $n\geq n_0$, there exists $t_n>0$ such that for
\begin{equation*}
x_n=\frac{1}{n}\,z_n-t_n\,Az_n
\end{equation*}
we have that
\begin{equation}
\label{x}
\|x_n\|=1.
\end{equation}

By (\ref{n}) and (\ref{x}) we get that
\begin{equation}
\label{t}
\lim_{n\rightarrow \infty}t_n=1. 
\end{equation}
By (\ref{n}), (\ref{nsquare}) and (\ref{t}) we get that 
\begin{equation*}
\lim_{n\rightarrow\infty}\|Az_n-n\,t_n\,A^2z_n\|=1
\end{equation*} 
and so 
\begin{equation}
\label{nA}
\lim_{n\rightarrow\infty}n\,\|Ax_n\|=1.
\end{equation}

Using the Cauchy--Schwarz inequality, (\ref{n}), (\ref{x}) and (\ref{nsquare}) we get that
\begin{equation*}
\frac{Re\langle Ax_n, x_n\rangle}{\|Ax_n\|}\leq\frac{\|\frac{1}{n}\,z_n\|}{n\,\|Ax_n\|}-\frac{t_n}{n\,\|Ax_n\|}+\frac{t_n}{n^2\,\|Ax_n\|}\,,
\end{equation*}
for all $n\geq n_0$. But, by (\ref{z}), (\ref{nA}) and (\ref{t}),
\begin{equation*}
\lim_{n\rightarrow\infty}\left(\frac{\|\frac{1}{n}\,z_n\|}{n\,\|Ax_n\|}-\frac{t_n}{n\,\|Ax_n\|}+\frac{t_n}{n^2\,\|Ax_n\|}\right)=-1.
\end{equation*}
Therefore $\phi(A)=\pi$, which is a contradiction. Hence
\begin{equation*}
X=R(A)\oplus N(A).
\end{equation*}
\end{proof}
%%%%%%%%%%%%%%%%%%%%%%%%%%%%%%%%%%%%%%%%%%%%%%%%%%%%%%%%%%%%%%%%%%%%%%%%%%%%

\begin{remark}
(i) In the proof of the previous theorem the properties of a Hilbert space which were used are that there exists a complement of $N(A)$ and that $\langle\cdot,\cdot\rangle$ is additive in the second variable. We don't know if the result of the theorem holds if we assume that $X$ is a Banach space and $N(A)$ is complemented in $X$.\\
(ii) If $X$ is a Hilbert space and $A$  is a bounded linear operator with closed range such that $\varphi(tA)<\pi$, for some $0\neq t\in\mathbb{C}$, then $X=R(A)\oplus N(A)$.
\end{remark}
%%%%%%%%%%%%%%%%%%%%%%%%%%%%%%%%%%%%%%%%%%%%%%%%%%%%%%%%%%%%%%%%%%%%%%%%%%%%

We will now discuss the case of a finite dimensional $X$. We will show that, under the additional assumption that $X$ is strictly convex, we can use the angle to characterize operators $A$ for which
\begin{equation*}
X=R(A)\oplus N(A).
\end{equation*}
%%%%%%%%%%%%%%%%%%%%%%%%%%%%%%%%%%%%%%%%%%%%%%%%%%%%%%%%%%%%%%%%%%%%%%%%%%%%

Recall that if $X$ is strictly convex and $x, y\in X$ with $x, y\neq 0$, then
\begin{equation*}
[x,y]=\|x\|\,\|y\|
\end{equation*}
implies that $y=\lambda x$, for some $\lambda\neq 0$ \cite[Theorem 5.1]{berkson}.
%%%%%%%%%%%%%%%%%%%%%%%%%%%%%%%%%%%%%%%%%%%%%%%%%%%%%%%%%%%%%%%%%%%%%%%%%%%%

\begin{theorem}
\label{characterization}
Let $X$ be a strictly convex, finite dimensional Banach space and $A:X\rightarrow X$ be a linear operator. Then
\begin{equation*}
X=R(A)\oplus N(A)
\end{equation*}
if and only if there exists $0\neq t\in\mathbb{C}$ such that $\varphi(tA)<\pi$.
\end{theorem}
%%%%%%%%%%%%%%%%%%%%%%%%%%%%%%%%%%%%%%%%%%%%%%%%%%%%%%%%%%%%%%%%%%%%%%%%%%%%

\begin{proof}
If $\varphi(tA)<\pi$, then, by Remark \ref{rem}(v) and (vi) we get that
\begin{equation*}
X=R(A)\oplus N(A).
\end{equation*}

For the converse we will first show that if $B:X\rightarrow X$ is a linear operator with
\begin{equation*}
X=R(B)\oplus N(B)\text{ and }\varphi(B)=\pi,
\end{equation*}
then $B$ has at least one negative eigenvalue.

Since $\varphi(B)=\pi$, there exists a sequence $(z_n)$ in $X$ such that $z_n\notin N(B)$ and $\|z_n\|=1$, for all $n\in\mathbb{N}$, and
\begin{equation}
\label{hyp}
\lim_{n\rightarrow\infty}\frac{Re[Bz_n, z_n]}{\|Bz_n\|}=-1.
\end{equation}

Since $X=R(B)\oplus N(B)$, for each $n\in\mathbb{N}$, there exist $0\neq x_n\in R(B)$ and $y_n\in N(B)$ such that
$z_n=x_n+y_n$. Since $X$ is finite dimensional we get that, by passing to subsequences denoted again by $(x_n)$ and $(y_n)$, there exist $x\in R(B)$ and $y\in N(B)$ such that
\begin{equation*}
\lim_{n\rightarrow\infty}x_n=x\text{ and }\lim_{n\rightarrow\infty}y_n=y.
\end{equation*}
We will prove that $x\neq 0$. Let $(e_i)_{i=1}^m$ be a basis of $R(B)$. Then, for each $1\leq i\leq m$, there exists a sequence $(a_n^i)$ in $\mathbb C$ such that
\begin{equation*}
x_n=\sum_{i=1}^ma_n^ie_i,
\end{equation*}
for all $n\in\mathbb{N}$. Let $c_n=\max_{1\leq i\leq m}|a_n^i|$, $n\in\mathbb{N}$ and $b_n^i=\frac{a_n^i}{c_n}$, $n\in\mathbb{N}$ and $i=1,2,...,m$. Then $c_n>0$, for all $n\in\mathbb{N}$, $|b_n^i|\leq 1$, for all $n\in\mathbb{N}$ and $i=1,2,...,m$, and 
\begin{equation*}
x_n=c_n\,\sum_{i=1}^mb_n^ie_i,
\end{equation*}
for all $n\in\mathbb{N}$. Moreover there exists $1\leq i_0\leq m$ such that $|b_n^{i_0}|=1$, for infinitely many $n$'s. Obviously, for each $i=1,2,...,m$, there exists a subsequence of $(b_n^i)$, which for simplicity we denote again by $(b_n^i)$, such that $\lim b_n^i=b_i$. We have that
\begin{eqnarray*}
\frac{Re[Bz_n, z_n]}{\|Bz_n\|}&=&\frac{Re[Bx_n, x_n+y_n]}{\|Bx_n\|}\\
&=&\frac{Re[B(b_n^{i_0}e_{i_0}+\sum_{i\neq i_0}b_n^ie_i), x_n+y_n]}{\|B(b_n^{i_0}e_{i_0}+\sum_{i\neq i_0}b_n^ie_i)\|}\,.
\end{eqnarray*}
Assume that $x=0$. Then $\|y\|=1$ and, by (\ref{hyp}),
\begin{equation}
\label{hyp1}
\frac{Re[B(b_{i_0}e_{i_0}+\sum_{i\neq i_0}b_ie_i), y]}{\|B(b_{i_0}e_{i_0}+\sum_{i\neq i_0}b_ie_i)\|}=-1.
\end{equation}
Let $w=b_{i_0}e_{i_0}+\sum_{i\neq i_0}b_ie_i\neq 0$. Then, by (\ref{hyp1}) and the Cauchy--Schwarz inequality, we get
\begin{equation*}
|[Bw, y]|=\|Bw\|.
\end{equation*}
Hence, as we mentioned above, the strict convexity of $X$ implies that $Bw$ and $y$ are linearly dependent and so $0\neq y\in R(B)$ which is a contradiction. Therefore $x\ne 0$. Hence $z=x+y\notin N(B)$, $\|z\|=1$ and
\begin{equation*}
Re[Bz, z]=-\|Bz\|.
\end{equation*}
As before we get that $Bz$ and $z$ are linearly dependent and $[Bz, z]$ is real. Hence $Bz=\lambda z$, for some $\lambda<0$ and thus $B$ has a negative eigenvalue.

To conclude the proof take $t\in \mathbb{C}$ such that $\lambda t$ is not a negative real number for all $\lambda\in\sigma(A)$. Then $tA$ has no negative eigenvalues and 
$$X=R(tA)\oplus N(tA)\,.$$ 
Hence by the previous part of the proof $\varphi(tA)<\pi$.
\end{proof}
%%%%%%%%%%%%%%%%%%%%%%%%%%%%%%%%%%%%%%%%%%%%%%%%%%%%%%%%%%%%%%%%%%%%%%%%%%%%

\begin{remark}
\label{bilateral}
(i) If $X$ is infinite dimensional, then $X=R(A)\oplus N(A)$ does not imply that there exists $0\neq t\in\mathbb{C}$ such that $\varphi(tA)<\pi$. To see that let $A$ be the bilateral shift on $l^2(\mathbb{Z})$. Then $A$ is unitary and $\sigma(A)$ is equal to the unit circle. Hence, for any $0\neq t\in \mathbb{C}$,
\begin{equation*}
\inf_{\|x\|=1}\frac{Re\langle (tA)x, x\rangle}{\|(tA)x\|}=-1.
\end{equation*}
Thus $\phi(tA)=\pi$, for all $t\neq 0$.\\
(ii) We don't know if the result of the above theorem holds if we omit the hypothesis that $X$ is strictly convex.
\end{remark}
%%%%%%%%%%%%%%%%%%%%%%%%%%%%%%%%%%%%%%%%%%%%%%%%%%%%%%%%%%%%%%%%%%%%%%%%%%%%

\section{Applications}
%%%%%%%%%%%%%%%%%%%%%%%%%%%%%%%%%%%%%%%%%%%%%%%%%%%%%%%%%%%%%%%%%%%%%%%%%%%%

Our first application is a simple proof of a theorem of N. Nirschl and H. Schneider \cite{nirschl}, \cite[Theorem 10.10]{bonsal}.
%%%%%%%%%%%%%%%%%%%%%%%%%%%%%%%%%%%%%%%%%%%%%%%%%%%%%%%%%%%%%%%%%%%%%%%%%%%%

\begin{theorem}[Nirschl and Schneider]
\label{nirschl}
Let $A:X\rightarrow X$ be a bounded linear operator with $0\in\partial\,\overline{co}\,V(A)$. Then $\alpha(A)\leq1$.
\end{theorem}
%%%%%%%%%%%%%%%%%%%%%%%%%%%%%%%%%%%%%%%%%%%%%%%%%%%%%%%%%%%%%%%%%%%%%%%%%%%%

\begin{proof}
As we already mentioned in the Preliminaries, $\partial\,\overline{co}\,V(A)=\partial\,\overline{co}\,W(A)$.
Since $0\in\partial\,\overline{co}\,W(A)$, there exists some $0\neq t\in \mathbb{C}$ such that $Re\,\lambda\geq 0$, for all $\lambda\in W(tA)$. Thus
\begin{equation*}
\inf\left\{Re[(tA)x,x]\,:\,\|x\|=1\right\}\geq 0
\end{equation*}
and so, by Corollary \ref{accretive}, we get that $\alpha(tA)\leq1$. Since $\alpha(tA)=\alpha(A)$, we have that $\alpha(A)\leq 1$.
\end{proof}
%%%%%%%%%%%%%%%%%%%%%%%%%%%%%%%%%%%%%%%%%%%%%%%%%%%%%%%%%%%%%%%%%%%%%%%%%%%%

In exactly the same manner, using Corollary \ref{accretive1} instead of Corollary \ref{accretive}, and $i(tA)=i(A)$, we can get an alternative proof of the following result of A. M. Sinclair \cite[Proposition 3]{sinclair}.
%%%%%%%%%%%%%%%%%%%%%%%%%%%%%%%%%%%%%%%%%%%%%%%%%%%%%%%%%%%%%%%%%%%%%%%%%%%%

\begin{proposition}[Sinclair]
\label{sinclair}
Let $A:X\rightarrow X$ be a bounded linear operator with closed range. If $0\in\partial\,\overline{co}\,V(A)$, then $i(A)\leq 1$.
\end{proposition}
%%%%%%%%%%%%%%%%%%%%%%%%%%%%%%%%%%%%%%%%%%%%%%%%%%%%%%%%%%%%%%%%%%%%%%%%%%%%

Our next result is an application of Theorem \ref{hilbert} and presents a sufficient condition for a bounded linear operator $A$ with closed range to satisfy
\begin{equation*}
X=R(A)\oplus N(A).
\end{equation*}
This condition involves the distance of the boundary of the numerical range from the origin.
%%%%%%%%%%%%%%%%%%%%%%%%%%%%%%%%%%%%%%%%%%%%%%%%%%%%%%%%%%%%%%%%%%%%%%%%%%%%

To proceed we need the following lemma.
%%%%%%%%%%%%%%%%%%%%%%%%%%%%%%%%%%%%%%%%%%%%%%%%%%%%%%%%%%%%%%%%%%%%%%%%%%%%

\begin{lemma}
\label{angleangle}
Let $X$ be a Hilbert space and $A:X\rightarrow X$ be a bounded linear operator such that $R(A)\perp N(A)$. If $\phi(A)>\frac{\pi}{2}$, then
\begin{equation*}
\phi(A)=\phi(A|_{N(A)^\perp}).
\end{equation*}
\end{lemma}
%%%%%%%%%%%%%%%%%%%%%%%%%%%%%%%%%%%%%%%%%%%%%%%%%%%%%%%%%%%%%%%%%%%%%%%%%%%%

\begin{proof}
Assume that $(z_n)$ is a sequence in $X$ such that $z_n\notin N(A)$, $\|z_n\|=1$, for all $n\in\mathbb{N}$, and
\begin{equation*}
\lim_{n\rightarrow\infty}\frac{Re\langle Az_n, z_n\rangle}{\|Az_n\|}=\cos A.
\end{equation*}
Then, for each $n\in\mathbb{N}$, there exist $x_n\in N(A)^\perp$ and $y_n\in N(A)$ such that
\begin{equation*}
z_n=x_n+y_n.
\end{equation*}
Since $R(A)\perp N(A)$,
\begin{equation}
\label{inequality}
\frac{Re\langle Az_n, z_n\rangle}{\|Az_n\|}=\frac{Re\langle Ax_n, x_n\rangle}{\|Ax_n\|\,\|x_n\|}\,\|x_n\|\geq \cos A\,\|x_n\|,
\end{equation}
for all $n\in\mathbb{N}$. Passing to a subsequence we may assume that $\displaystyle{\lim_{n\rightarrow\infty}}\|x_n\|=1$. 
Indeed if we assume the contrary, i.e.\ that $\displaystyle{\lim_{n\rightarrow\infty}}\|x_n\|<1$, using $\phi(A)>\frac{\pi}{2}$, we get $\cos A\,\displaystyle{\lim_{n\rightarrow\infty}}\|x_n\|>\cos A$, which contradicts (\ref{inequality}). So
\begin{equation*}
\cos \left(A|_{N(A)^\perp}\right)\leq\lim_{n\rightarrow\infty}\frac{Re\langle Ax_n, x_n\rangle}{\|Ax_n\|\,\|x_n\|}=\cos A.
\end{equation*}
On the other hand, by definition,
\begin{equation*}
\cos \left(A|_{N(A)^\perp}\right)\geq \cos A,
\end{equation*}
and so $\cos \left(A|_{N(A)^\perp}\right)=\cos A$. Hence
\begin{equation*}
\phi(A|_{N(A)^\perp})=\phi(A).
\end{equation*}
\end{proof}
%%%%%%%%%%%%%%%%%%%%%%%%%%%%%%%%%%%%%%%%%%%%%%%%%%%%%%%%%%%%%%%%%%%%%%%%%%%%

\begin{remark}
Lemma \ref{angleangle} does not hold if $\varphi(A)\leq\frac{\pi}{2}$. To see that let $A$ be an orthogonal projection with $A\neq 0, I_X$. Then $R(A)\perp N(A)$ and
\begin{equation*}
\phi(A)=\frac{\pi}{2}>\phi(A|_{N(A)^\perp})=0.
\end{equation*}
\end{remark}
%%%%%%%%%%%%%%%%%%%%%%%%%%%%%%%%%%%%%%%%%%%%%%%%%%%%%%%%%%%%%%%%%%%%%%%%%%%%

If $A$ has closed range, by $c_A$ we denote the largest positive constant $c$ such that $\|Ax\|\geq c\|x\|$, for all $x\in N(A)^\perp$.
%%%%%%%%%%%%%%%%%%%%%%%%%%%%%%%%%%%%%%%%%%%%%%%%%%%%%%%%%%%%%%%%%%%%%%%%%%%%

Our result is the following.
%%%%%%%%%%%%%%%%%%%%%%%%%%%%%%%%%%%%%%%%%%%%%%%%%%%%%%%%%%%%%%%%%%%%%%%%%%%%

\begin{theorem}
\label{numrange}
Let $X$ be a Hilbert space, $A:X\rightarrow X$ be a bounded linear operator with closed range such that $R(A)+N(A)$ is closed and
\begin{equation*}
R(A)\cap N(A)=\{0\}\,.
\end{equation*} 
Let
\begin{equation*}
\rho_A=c_A\cdot\gamma(R(A),N(A)).
\end{equation*}
If
\begin{equation*}
\partial W(A)\cap B(0, \rho_A)\neq\emptyset,
\end{equation*}
then $X=R(A)\oplus N(A)$.
\end{theorem}
%%%%%%%%%%%%%%%%%%%%%%%%%%%%%%%%%%%%%%%%%%%%%%%%%%%%%%%%%%%%%%%%%%%%%%%%%%%%

\begin{proof}
First of all note that, since $R(A)$ is closed, $c_A>0$, and, since $R(A)+N(A)$ is closed, $\gamma(R(A), N(A))>0$. Thus $\rho_A>0$. 

Moreover note that, since $R(A)\cap N(A)=\{0\}$,
\begin{equation}
\label{gap1}
\gamma(R(A),N(A))=\inf_{0\neq x\in R(A)}\frac{\mathrm{dist}(x,N(A))}{\|x\|}\,.
\end{equation}

If
\begin{equation*}
0\neq\lambda=\rho e^{i\theta}\in \partial W(A)\cap B(0, \rho_A),
\end{equation*}
then $0<\rho<\rho_A$ and hence if we multiply $A$ by $t=e^{i(\pi-\theta)}$ we get that
\begin{equation*}
\inf_{\|x\|=1} Re\langle (tA)x,x\rangle=-\rho>-\rho_A.
\end{equation*}
Hence without loss of generality we may assume that
\begin{equation}
\label{ep}
-\rho_A<\inf_{\|x\|=1} Re\langle Ax,x\rangle<0.
\end{equation}

Let $P$ be the orthogonal projection onto $N(A)^\perp$. Since $R(A)+N(A)$ is closed, by \cite[Theorem 2.1]{schochetman},
\begin{equation*}
P|_{R(A)}:R(A)\rightarrow N(A)^\perp
\end{equation*}
has closed range. Moreover, since $R(A)\cap N(A)=\{0\}$, it is injective.

Let $B:X\rightarrow X$ with $B=PA$. Since 
\begin{equation*}
P|_{R(A)}:R(A)\rightarrow N(A)^\perp
\end{equation*}
is injective we have that $N(B)=N(A)$. Also since
\begin{equation*}
P|_{R(A)}:R(A)\rightarrow N(A)^\perp
\end{equation*}
has closed range, the operator $B$ has closed range. Finally, it is obvious that $R(B)\perp N(B)$.

We will show that $X=R(B)\oplus N(B)$.

If $\phi(B)\leq\frac{\pi}{2}$, then, by Theorem \ref{theorem},
\begin{equation*}
X=R(B)\oplus N(B).
\end{equation*}

On the other hand, if $\phi(B)>\frac{\pi}{2}$, then, by Lemma \ref{angleangle}, we get that
\begin{equation*}
\phi(B)=\phi(B|_{N(B)^\perp}).
\end{equation*}
By (\ref{gap1}),
\begin{equation*}
\|PAx\|=d(Ax,N(A))\geq\gamma(R(A),N(A))\|Ax\|,\text{ for all }x\in X.
\end{equation*}
Thus
\begin{equation}
\label{gap}
\|PAx\|\geq\rho_A\|x\|,\text{ for all }x\in N(A)^\perp.
\end{equation}
Take $x\in N(B)^\perp=N(A)^\perp$ such that $\|x\|=1$ and  $Re\langle Bx, x\rangle<0$. Then using (\ref{ep}) and (\ref{gap}) we get that
\begin{equation*}
\frac{Re\langle Bx, x\rangle}{\|Bx\|}=\frac{Re\langle Ax, x\rangle}{\|PAx\|}\geq\frac{Re\langle Ax, x\rangle}{\rho_A}\geq\frac{\inf_{\|y\|=1}Re\langle Ay, y\rangle}{\rho_A}>-1.
\end{equation*}
Hence
\begin{equation*}
\phi(B)=\phi(B|_{N(B)^\perp})<\pi.
\end{equation*}
Using Theorem \ref{theorem} we get that
\begin{equation*}
X=R(B)\oplus N(B).
\end{equation*}

To conclude the proof note that the operator
\begin{equation*}
U=\left(
\begin{array}{ccl}
(P|_{R(A)})^{-1} & 0\\
0 & I
\end{array}
\right),
\end{equation*}
with respect to the decomposition $X=R(B)\oplus N(B)$,
is one-to-one and onto $X$. Thus
\begin{equation*}
X=R(A)\oplus N(A).
\end{equation*}
\end{proof}
%%%%%%%%%%%%%%%%%%%%%%%%%%%%%%%%%%%%%%%%%%%%%%%%%%%%%%%%%%%%%%%%%%%%%%%%%%%%

We will now deal with  operators whose spectrum does not intersect all rays emanating from the origin. Note that we will use again Lemma \ref{angleangle}.
%%%%%%%%%%%%%%%%%%%%%%%%%%%%%%%%%%%%%%%%%%%%%%%%%%%%%%%%%%%%%%%%%%%%%%%%%%%%

\begin{definition}
By a  ray emanating from the origin we mean the set
\begin{equation*}
R_\omega=\left\{0\neq t\in \mathbb C:\, \arg t=\omega\right\},
\end{equation*}
for some $\omega\in[0, 2\pi)$.
\end{definition}
%%%%%%%%%%%%%%%%%%%%%%%%%%%%%%%%%%%%%%%%%%%%%%%%%%%%%%%%%%%%%%%%%%%%%%%%%%%%

To proceed we need the following lemma.
%%%%%%%%%%%%%%%%%%%%%%%%%%%%%%%%%%%%%%%%%%%%%%%%%%%%%%%%%%%%%%%%%%%%%%%%%%%%

\begin{lemma}
\label{spectrum}
Let $X$ be a Hilbert space and $A:X\rightarrow X$ be a bounded linear operator with closed range such that $R(A)\perp N(A)$. If $\varphi(A)=\pi$, then
\begin{equation*}
\sigma(A)\cap R_\pi\neq \emptyset.
\end{equation*}
\end{lemma}
%%%%%%%%%%%%%%%%%%%%%%%%%%%%%%%%%%%%%%%%%%%%%%%%%%%%%%%%%%%%%%%%%%%%%%%%%%%%

\begin{proof}
We will show that 
\begin{equation*}
\sigma(A)\cap R_\pi=\emptyset
\end{equation*}
implies that $\varphi(A)<\pi$.

Assume that
\begin{equation*}
\sigma(A)\cap R_\pi=\emptyset.
\end{equation*}
Then $A+tI$ is invertible for all $t>0$.

First suppose that $N(A)=\left\{0\right\}$. 

If $\varphi(A)=\pi$, then we can find a sequence $(x_n)$ in $X$, with $\|x_n\|=1$, for all $n\in\mathbb{N}$, such that
\begin{equation*}
\lim_{n\rightarrow\infty}\frac{Re\langle Ax_n,x_n\rangle}{\|Ax_n\|}=-1.
\end{equation*}
Since $A$ has closed range there exists $c>0$ such that $\|Ax_n\|\geq c$. Hence, by passing to a subsequence if necessary, we get that
\begin{equation*}
\lim_{n\rightarrow \infty}Re\langle Ax_n, x_n\rangle=-t_0\text{ and }\lim_{n\rightarrow \infty}\|Ax_n\|=t_0,
\end{equation*}
for some $t_0>0$. Without loss of generality we may assume that $t_0=1$ (if not take $\frac{1}{t_0}A$ instead of $A$). We have that
\begin{equation*}
\frac{1}{2}\,|Re\langle Ax_n, x_n\rangle-1|\leq\|\frac{Ax_n-x_n}{2}\|\leq\frac{1}{2}\,(\|Ax_n\|+1),
\end{equation*}
for all $n\in \mathbb N$ and so $\displaystyle{\lim_{n\rightarrow \infty}}\|\frac{Ax_n-x_n}{2}\|=1$. Since $X$ is a Hilbert space it is uniformly convex and thus using \cite[Proposition 5.2.8 (d)]{Megginson} we get that $\displaystyle{\lim_{n\rightarrow \infty}}\|Ax_n+x_n\|=0$ which contradicts the invertibility of $A+I$. Hence $\varphi(A)<\pi$.

For the case where $N(A)\neq\left\{0\right\}$, following the same steps as above we may show that
\begin{equation*}
\varphi(A|_{N(A)^\perp})<\pi.
\end{equation*}
Hence, using Lemma \ref{angleangle}, we get that
\begin{equation*}
\varphi(A)=\varphi(A|_{N(A)^\perp})<\pi.
\end{equation*}
\end{proof}
%%%%%%%%%%%%%%%%%%%%%%%%%%%%%%%%%%%%%%%%%%%%%%%%%%%%%%%%%%%%%%%%%%%%%%%%%%%%

We can now prove that if $R(A)\perp N(A)$ and $\sigma(A)\cap R_\omega= \emptyset$, for some $\omega\in[0, 2\pi)$, then $X=R(A)\oplus N(A)$.
%%%%%%%%%%%%%%%%%%%%%%%%%%%%%%%%%%%%%%%%%%%%%%%%%%%%%%%%%%%%%%%%%%%%%%%%%%%%

\begin{theorem}
\label {ray}
Let $X$ be a Hilbert space and $A:X\rightarrow X$ be a bounded linear operator with closed range such that $R(A)\perp N(A)$. If
\begin{equation*}
\sigma(A)\cap R_\omega= \emptyset,
\end{equation*}
for some $\omega\in[0, 2\pi)$, then $X=R(A)\oplus N(A)$.
\end{theorem}
%%%%%%%%%%%%%%%%%%%%%%%%%%%%%%%%%%%%%%%%%%%%%%%%%%%%%%%%%%%%%%%%%%%%%%%%%%%%

\begin{proof}
Since
\begin{equation*}
\sigma(A)\cap R_\omega= \emptyset,
\end{equation*}
for $ t=e^{i(\pi-\omega)}$ we have that $\sigma(tA)\cap R_\pi=\emptyset$. Hence, by Lemma \ref{spectrum}, we get that $\varphi(tA)<\pi$, which, by Theorem \ref{hilbert}, implies that
\begin{equation*}
X=R(tA)\oplus N(tA)=R(A)\oplus N(A).
\end{equation*}
\end{proof}
%%%%%%%%%%%%%%%%%%%%%%%%%%%%%%%%%%%%%%%%%%%%%%%%%%%%%%%%%%%%%%%%%%%%%%%%%%%%
 
Note that in the first part of the proof of Lemma \ref{spectrum}, the crucial property of the Hilbert space $X$ was its uniform convexity. Hence this part may be adapted in the case of a uniformly convex Banach space. Adding this hypothesis, we may now prove an interesting property of operators with closed range, whose spectrum does not intersect all rays emanating from the origin: they are surjective if and only if they are injective.
%%%%%%%%%%%%%%%%%%%%%%%%%%%%%%%%%%%%%%%%%%%%%%%%%%%%%%%%%%%%%%%%%%%%%%%%%%%%

\begin{theorem}
Assume that both $X$ and $X^\ast$ are uniformly convex Banach spaces and $A:X\rightarrow X$ is a bounded linear operator with closed range such that
\begin{equation*}
\sigma(A)\cap R_\omega= \emptyset,
\end{equation*}
for some $\omega\in[0, 2\pi)$. Then $A$ is surjective if and only if it is injective.
\end{theorem}
%%%%%%%%%%%%%%%%%%%%%%%%%%%%%%%%%%%%%%%%%%%%%%%%%%%%%%%%%%%%%%%%%%%%%%%%%%%%

\begin{proof}
Assume that $A$ is injective. Using the same arguments as those in the first part of the proof of Lemma \ref{spectrum} and arguing as above we get that $\varphi(tA)<\pi$, for some $0\neq t\in\mathbb C$. So by Theorem \ref{theorem} we get that $X=R(A)$.

Conversely if $A$ is surjective, then $A^\ast$ is injective and, since $\sigma(A^\ast)=\sigma(A)$, we get, again as above, that $A^\ast$ is surjective and so $A$ is injective.
\end{proof}
%%%%%%%%%%%%%%%%%%%%%%%%%%%%%%%%%%%%%%%%%%%%%%%%%%%%%%%%%%%%%%%%%%%%%%%%%%%%

\bibliographystyle{amsplain}

%%%%%%%%%%%%%%%%%%%%%%%%%%%%%%%%%%%%%%%%%%%%%%%%%%%%%%%%%%%%%%%%%%%%%%%%%%%%

\end{document}